\documentclass[a4paper,12pt]{article}
\usepackage[english]{babel}
\usepackage[T1]{fontenc}
\usepackage{amsthm,enumerate,amsmath,amscd,amssymb,verbatim}
\usepackage{graphicx}

\newcommand{\R}{\mathbb{R}}
\newcommand{\Rn}{\mathbb{R}^n}
\newcommand{\p}{\varphi}
\newcommand{\e}{\delta}
\newcommand{\E}{\Delta}
\newcommand{\diam}{\textnormal{diam} \,}

\font\fFt=eusm10 
\font\fFa=eusm7  
\font\fFp=eusm5  
\def\K{\mathchoice
{\hbox{\,\fFt K}}
{\hbox{\,\fFt K}}
{\hbox{\,\fFa K}}
{\hbox{\,\fFp K}}}

\newcounter{minutes}
\divide\time by 60
\newcounter{hours}
\multiply\time by 60
\addtocounter{minutes}{-\time}

\swapnumbers
\numberwithin{equation}{section}
\theoremstyle{plain}

\newtheorem{theorem}[equation]{Theorem}
\newtheorem{corollary}[equation]{Corollary}
\newtheorem{proposition}[equation]{Proposition}
\newtheorem{lemma}[equation]{Lemma}

\newtheorem{remark}[equation]{Remark}

\newcounter {own}
\def\theown {\thesection       .\arabic{own}}

\newenvironment{nonsec}{\bf
\setcounter{own}{\value{equation}}\addtocounter{equation}{1}
\refstepcounter{own}
\medskip

\noindent \theown.$\ \,$}{\!\!. }

\begin{document}

\begin{center}
{\bf \large Teichm\"uller's problem in space}
\end{center}

\begin{center}
{\bf R. Kl\'en, V. Todor\v{c}evi\'c and M. Vuorinen}
\end{center}


\def\thefootnote{}
\footnotetext{
\texttt{\tiny File:~\jobname .tex,
          printed: \number\year-\number\month-\number\day,
          \thehours.\ifnum\theminutes<10{0}\fi\theminutes}
}
\makeatletter\def\thefootnote{\@arabic\c@footnote}\makeatother

{\bf Abstract:}  Quasiconformal homeomorphisms of the whole space
${\mathbb R}^n,$  onto itself normalized at one or two points are studied.
In particular, the stability theory, the case when the maximal dilatation tends to $1\,,$ is
in the focus. Our main result provides a spatial analogue of a classical result due to Teichm\"uller. Unlike Teichm\"uller's result, our bounds are explicit.
Explicit bounds are based on two sharp well-known distortion results:
the quasiconformal Schwarz lemma and the bound for linear dilatation.
Moreover, Bernoulli type inequalities and asymptotically sharp bounds for special functions involving complete elliptic integrals are applied to simplify the computations.
Finally, we discuss  the behavior
of the quasihyperbolic metric under quasiconformal maps and prove a sharp result for quasiconformal maps of
${\mathbb R}^n \setminus \{0\}$  onto itself.


\section{Introduction}

One of the main topics of the theory of $K$-quasiconformal maps of subdomains of the Euclidean space $\Rn$, $n \ge 2$, deals with the behavior of this class of mappings when $K \to 1$. For higher dimensions $n \ge 3$ a surprising limiting behavior occurs: when $K=1$, the class of $K$-quasiconformal maps coincides with the class of M\"obius transformations. This is the content of a classical theorem due to Liouville, for the case of smooth maps. In his pioneering work, Yu.~G.~Reshetnyak \cite{r} gave estimates for the distance of a given $K$-quasiconformal map from the closest M\"obius transformation in the sense of suitable norm when $|K-1|$ is small and gave the name
"stability theory" for this research area. Reshetnyak's work provides, among other things, a far-reaching generalization of Liouville's theorem.  As a sample result  of Reshetnyak's deep work we formulate  \cite[Lemma 2.9]{r} where the stability estimate is expressed in terms of a function  $\mu\,,$ known only qualitatively.

\begin{theorem} {\rm (\cite[Lemma 2.9]{r})} \label{resh}
There are a number $q$, $0<q<1$, and a nondecreasing function $\mu \colon (0,\infty) \to \R$ such that $\mu(t) \to 0$ as $t \to 0$ and, for every mapping $f \colon B^n(a,r) \to \Rn$ with distortion  bounded by $K$, we can indicate a M\"obius transformation $g$ for which
\[
  |x-g^{-1}(f(x))| \le r \mu(K-1)
\]
for all $x \in B^n(a,qr)$.
\end{theorem}

\medskip

Note that a mapping is quasiconformal if and only if it is injective and of bounded distortion.
The proof of Theorem \ref{resh}
makes use of normal family arguments and does not seem to give explicit quantitative bound for the function $\mu(t)\,.$

Our goal here is to study the stability theory and to establish quantitative explicit bounds with concrete constants that enable us to estimate the distance of a normalized $K$-quasiconformal map from the identity map in terms of $K-1$ and the dimension $n$. We consider the standard normalization which requires that the mapping keeps two points fixed and prove a stability result for dimensions $n \ge 2\,,$ which is a counterpart  O. Teichm\"uller's
result in the  case $n=2$.

For the statement and formulation of our results we introduce some necessary notation.
For $n \ge 2, K \ge 1,$ let
$$
QC_K({\mathbb R }^n) = \{f\,:\,\mathbb R^n \to \mathbb R^n \,: \,f \,\mbox{is}\,\,K-\mbox{
quasiconformal } \}\,.$$
Here and through the paper we use the standard definition of
$K$-quasicon\-formality from
\cite{v1}. It is a well-known basic fact that a map $f \in QC_K({\mathbb
R }^n)$ has a
homeomorphic extension to $\infty$ with $f(\infty)= \infty\,$ (in fact,
this can be also seen from Proposition \ref{qcestimate} below). Thus, $f$
is defined in the M\"obius space
$\overline{\mathbb R^n} = \mathbb R^n \cup \{ \infty \} \,.$ Without
further remark we
always assume that our maps are extended in this way. For the sake of
convenience, we
will consider a subclass of $
QC_K({\mathbb R }^n)$ consisting of maps normalized at two finite points (plus the above normalization at infinity)
as follows
\[
M_K({\mathbb R }^n) = \{f \in QC_K({\mathbb R }^n) \colon
f(x)=x,\quad\forall
x\in \{0, e_1, \infty\} \} \,,
\]
where $e_1=(1,0,...,0)$ is the first unit vector.

In his classical work \cite{t} O. Teichm\"uller studied the class $
M_K({\mathbb R }^2)$ and proved the following inequality for the
hyperbolic metric
$s_G$ of $G= {\mathbb R }^2 \setminus \{ 0, e_1\}:$ If $f \in
M_K({\mathbb R }^2)\,,$ then the sharp inequality
$$
K \ge \exp(s_G(x,f(x)))
$$
holds for all $x \in G\,.$ For the definition of the hyperbolic metric see \cite{kl}.

This result may be considered as a stability result: $f(x)$ is contained
in the ball
$B_{s_G}(x, \log K)$ of the metric $s_G$ centered at the point $x$ and
with the radius $\log K\,.$ In particular, for $K \to 1\,,$ the radius
tends to $0\,.$

On one hand this result is sharp, on the other hand the information it
provides is
implicit. Indeed the geometric structure of the balls $B_{s_G}(x,r)$ has
not been
carefully studied to our knowledge and even finding a useful upper bound for its chordal
diameter in terms of $x, r$ is not known to us.
Furthermore, for a general plane domain $D\,,$
it is a basic fact that the shapes of the balls $B_{s_D}(z,t)$ depend very much on the
geometric structure (and "thickness/thinness") of the
boundary $\partial D\,,$ the center $z$ and the radius $t\,,$ but quantitative
estimates are hard to find in the literature. 

It is easy to see that for a fixed $t > 0$ there is a diversity of shapes of the disks $B_{s_D}(z,t)$: they need not be homeomorphic to each other. For instance if $D = \R^2 \setminus \{ 0,1 \}$ it may happen that $B_{s_D}(z_1,t)$ is homeomorphic to the disk whereas $B_{s_D}(z_2,t)$ is homeomorphic to an annulus for some $z_2 \neq z_1$ and $t >0\,.$

Because of our interest in explicit stability estimates it seems to be a natural question to study Teichm\"uller's
result in the context of metric spaces equipped with a metric  more concrete than
the hyperbolic metric. An example of such a metric is the distance ratio
metric or the $j$-metric studied below. Our goal here is to study the extent to which
 Teichm\"uller's result can be generalized to ${\mathbb R }^3$.
Because a rotation around the $e_1$-axis leaves the whole $e_1$-axis and
in particular the triple $\{0,e_1,\infty\}$ fixed, we see that for a
fixed $x \in {\mathbb R }^3 \setminus \{ 0, e_1\}$ and $K\ge 1$ the set
\[
  V_K(x) = \{ f(x): f\in M_K({\mathbb R }^3) \}
\]
has rotational symmetry, it is a solid of revolution
with the $e_1$-axis as the symmetry axis. Moreover,  by Theorem \ref{main1} below, for a fixed $x$ when $K \to 1\,,$
this solid of revolution $ V_K(x)$ converges to a circle centered on the $e_1$-axis and perpendicular to the $e_1$-axis. If we look at the cross section of $ V_K(x)$ with a plane
containing $e_1$-axis, then Theorem \ref{main1} shows that the diameter of this cross section indeed does tend to $0$ with a quantitative upper bound in terms of $K-1$.

Very recently several authors have studied various extensions, ramifications and generalizations of Teichm\"uller's
problem. For the case of the unit ball in ${\mathbb R}^n$ see \cite{mv,vz} and for the case of Riemann surfaces
see \cite{btct}.

In Section \ref{preliminary} we review the necessary background information and
prove some auxiliary lemmas including Bernoulli type inequalities that will be used in the proof of the main results.

We next introduce our three main results. In the first result we study the cross section of $ V_K(x)$ with a plane
containing $e_1$-axis.

\begin{theorem} \label{main1}
  Fix $x \in \mathbb{R}^n \setminus\{0,e_1\}\,,$  let
  $$
  \Delta(x) = \min \{ |x|+|x-e_1|-1, 1 -||x|-|x-e_1|  |   \}/3\,,
  $$
and let  $\e \in (0,\E(x) )$ and 
$$K <1+ \left( (\log (\e/2+1))/62 \right)^2 \,.$$ 
If $f \in  QC_K(\R^n)$ with $f(z)=z$ for $z \in  \{0, e_1\}\,,$ then
there exists a M\"obius transformation, a rotation $h :\R^n \to \R^n$
around the $e_1$-axis, such that
$$
|h(f(x))-x| \le 2 \pi \max \{|x|,|x-e_1| \} \sqrt{\e} \,.
$$

\end{theorem}

Theorem \ref{main1} provides information
about the size of the set $V_K(x)$ when $K>1$ is close to $1.$ The proof
of this theorem combines a number of ideas. The first idea is to
show that the images of both of the spheres $S^{n-1}(0,|x|)$ and
$S^{n-1}(e_1,|x-e_1|)$ are contained in spherical annular domains, centered
at 0 and $e_1$, respectively, with
a good control and explicit bounds for the inner and outer radii of the annuli in each case. Therefore $V_K(x)$ is a subset of the intersection of these two annuli and it remains to estimate the size of the intersection in terms of radii. For simplicity
we require that the  intersection of the annuli does not contain points of the $e_1$-axis.
We give a sufficient condition for this requirement which we call the annuli intersection
criterion. Roughly speaking this criterion says that both boundary spheres of one of the annuli
intersect both boundary spheres of the other annulus.
In fact, the number $\E(x)$ in  Theorem \ref{main1} is connected with the annuli intersection criterion.

For the statement of our second main result, Theorem \ref{main2},
we introduce the so called distance ratio metric. Also here an asymptotically
sharp result is proved.
The \emph{distance ratio metric} or \emph{$j$-metric} in a proper subdomain $G$ of the Euclidean space $\Rn$, $n \ge 2$, is defined by
$$
  j_G(x,y) = \log \left( 1+\frac{|x-y|}{\min \{ d(x),d(y) \}}\right),
$$
where $d(x)$ is the Euclidean distance between $x$ and $\partial G$.
A slightly different definition of $ j_G$ was applied in \cite{gp} and the present
form of the definition stems from \cite{vu1}.
Asymptotically sharp explicit results, such as  Theorem \ref{main2}, are very few
in the literature on quasiconformal maps in ${\mathbb R}^n\,.$ The proof of
Theorem \ref{main2} relies on two ingredients: a sharp version of the Schwarz lemma for quasiconformal maps from \cite{avv2} and a sharp bound for the linear dilatation from \cite{vu2}. Theorem \ref{main2} appears to be new even for the case $n=2\,.$ We expect that Teichm\"uller's work cited above might be used to
prove this type of results, but the technical obstacles seem to be significant.

\begin{theorem}  \label{main2}
  Let $G = \Rn \setminus \{ 0 \}$, $f \in QC_K(\mathbb{R}^n)$, $K \in (1,2]$ and $f(0) = 0\,.$  There exists $c(K)$ such that for all $x,y \in G$
  \[
    j_G(f(x),f(y)) \le c(K) \max \{ j_G(x,y)^\alpha , j_G(x,y) \},
  \]
where $\alpha = K^{1/(1-n)}$, and $c(K) \to 1$ as $K \to 1$.
\end{theorem}

Our third main result yields a counterpart of Theorem  \ref{main2} for the quasihyperbolic metric.
The {\em quasihyperbolic metric} of $G$ is defined by
the quasihyperbolic length minimizing property
\begin{equation}\label{qhmetric}
k_G(x,y)=\inf_{\gamma\in \Gamma(x,y)} \ell_k(\gamma),
\quad \ell_k(\gamma) =\int_\gamma \frac{|dz|}{d(z)}\,,
\end{equation}
where $\ell_k(\gamma)$ is the quasihyperbolic length of
$\gamma$ (cf. \cite{gp}) and $d(z)$ stands for the distance $d(z,\partial G)$ of
$z \in G$ to the boundary. It is well-known that \cite[Lemma 2.1]{gp}
$$
k_G(x,y) \ge j_G(x,y)
$$
for all $x,y \in G\,.$ Gehring and Osgood proved the following quasi-invariance property
of the quasihyperbolic metric.

\begin{theorem} ( \cite[Theorem 3]{go}) \label{gothm}
For $n \ge 2, K\ge1\,,$ there exists a constant $c=c(n,K)$ depending only on $n$ and $K$ such that the following
holds.
 If $f:G\to G'$ is a $K-$quasiconformal homeomorphism between domains $G,G'\subset \mathbb{R}^n$ and $x,y\in G\,,$ then
\[
k_{G'}(f(x)f(y))\leq \max\{k_G(x,y)^{\alpha},k_G(x,y)\}, \, \alpha=K^{1/(1-n)}.
\]
\end{theorem}

 It can be easily shown
that the quasihyperbolic metric is not invariant under M\"obius transformations of the
unit ball onto itself and hence the constant $c$ in Theorem \ref{gothm} cannot be
asymptotically sharp when $K \to 1\,.$ In other words, $c \nrightarrow 1$ when $K \to 1\,.$
In the third main result, Theorem \ref{main3}, we show that, however, for the
special domain  $G= {\mathbb R}^n \setminus \{ 0\}$ we have a result with the
asymptotically sharp constants. This result is new also in the case $n = 2$.

\begin{theorem}\label{main3}
For given $K \in (1,2]$ and $n\ge 2$ there exists a constant $\omega(K,n)$
such that if $G= {\mathbb R}^n \setminus \{ 0\}$ and $f: {\mathbb R}^n \to
{\mathbb R}^n $ is a $K$-quasiconformal mapping with $f(0)=0\,,$ then
for all $x,y\in G$
\[k_G(f(x),f(y)) \le \omega(K,n) \max \{ k_G(x,y)^{\alpha},  k_G(x,y) \}
\]
where $\alpha= K^{1/(1-n)}$ and $\omega(K,n) \to 1$ when $K \to 1\,.$
\end{theorem}


\section{Preliminary results}\label{preliminary}

When $s > 1$ for the \emph{Gr\"otzsch capacity} we use the notation
$ \gamma_n(s)$ as in \cite[p. 88]{vu1}. Then for the planar case we
have by \cite[p.66]{vu1} $\gamma_2(s) = 2\pi / \mu(1/s)$, where
\[
  \mu (r) = \frac{\pi}{2}\frac{\K(\sqrt{1-r^2})}{\K(r)} \quad \textnormal{and} \quad \K(r) = \int_0^1 \frac{dx}{\sqrt{(1-x^2)(1-r^2x^2)}}
\]
for $r \in (0,1)$. We define for $r \in (0,1)$ and $K > 0$
\[
  \p_{K,n}(r) = \frac{1}{\gamma_n^{-1}(K\gamma_n(1/r))}.
\]

Let $\eta \colon [0,\infty) \to [0,\infty)$ be an increasing homeomorphism and $D,D' \subset \Rn$. A homeomorphism $f \colon D \to D'$ is \emph{$\eta$-quasisymmetric} if
\begin{equation}\label{qsdefinition}
  \frac{|f(a)-f(c)|}{|f(b)-f(c)|} \le \eta \left( \frac{|a-c|}{|b-c|} \right)
\end{equation}
for all $a,b,c \in D$ and $c \ne b$. By \cite{v2} a $K$-quasiconformal
mapping of the whole space $\Rn$ is $\eta_{K,n}$-quasisymmetric with a
control function $\eta_{K,n}$. Let us define the optimal control
function by
\[
  \eta^*_{K,n}(t) = \sup \{ |f(x)| \colon |x| \le t, f \in QC_K(\Rn), f(y) = y \textnormal{ for } y \in \{ 0,e_1,\infty \} \}.
\]

Vuorinen \cite[Theorem 1.8]{vu2} proved an upper bound for
$\eta^*_{K,n}(t)$, which was later refined in  \cite[Theorem 14.8]{avv} for $K \ge 1$ into the following form
\begin{equation} \label{etabd}
  \eta^*_{K,n}(t) \le \left\{ \begin{array}{ll}
    \displaystyle \eta^*_{K,n}(1)\p_{K,n}(t), & 0 < t < 1,\\
    \displaystyle  \exp(4K(K+1)\sqrt{K-1}) , & t = 1,\\
    \displaystyle \eta^*_{K,n}(1)\frac{1}{\p_{1/K,n}(1/t)}, & t > 1.
  \end{array} \right.
\end{equation}

These bounds could be further refined (see \cite[14.36(4)]{avv}, \cite{p}, \cite{s0}), but we use simpler bounds. A simplified, but still asymptotically sharp upper bound for
 $\eta^*_{K,n}(t)$ can be written as follows
\begin{equation}\label{roughupperbound}
  \eta^*_{K,n}(t) \le \left\{ \begin{array}{ll}
    \displaystyle \eta^*_{K,n}(1)\lambda_n^{1-\alpha}t^\alpha, & 0 < t \le 1,\\
    \displaystyle \eta^*_{K,n}(1)\lambda_n^{\beta-1}t^{\beta}, & t > 1,
  \end{array} \right.
\end{equation}
where $\alpha = K^{1/(1-n)}$ and $\beta = 1/\alpha$. Furthermore, by \cite[Lemma 7.50]{vu1} we have the following inequalities
\begin{equation}\label{upperboundsoflambda}
  \lambda_n^{1-\alpha} \le 2^{1-1/K}K \quad \textnormal{and} \quad \lambda_n^{1-\beta} \ge 2^{1-K}K^{-K}.
\end{equation}

Recently refined versions of \eqref{etabd}  have been applied by M. Badger, J.T. Gill, S. Rohde and T. Toro \cite{bgrt} and
I. Prause \cite{p} to find upper bounds for the Hausdorff dimension of quasiconformal images of spheres.

\begin{proposition}\label{qcestimate}
  Let $K \in (1,2]$, $f \in QC_K(\Rn)$, $f(x)=x$ for $x \in \{ 0,e_1 \}$, $\alpha = K^{1/(1-n)}$ and $\beta = 1/\alpha$. Then
  \[
    \begin{array}{ll}
      \displaystyle \frac{1}{c_3}|x|^\beta \le |f(x)| \le c_3 |x|^\alpha, & {if }\,\, 0<|x|\le1,\\
      \displaystyle \frac{1}{c_3}|x|^\alpha \le |f(x)| \le c_3 |x|^\beta, & {if } \,\, |x|>1
    \end{array}
  \]
  for $c_3 = \exp (60 \sqrt{K-1})$.
\end{proposition}
\begin{proof}
  Since $f$ is quasiconformal it is also $\eta^*_{K,n}$-quasisymmetric and by choosing $a=x$, $b=0$ and $c=e_1$ in (\ref{qsdefinition}) we have $|f(x)| \le \eta^*_{K,n}(|x|)$. Similarly, selection $(a,b,c) = (e_1,x,0)$ in (\ref{qsdefinition}) gives $|f(x)| \ge 1/\eta^*_{K,n}(1/|x|)$. Therefore
  \begin{equation}\label{etabounds}
    \frac{1}{\eta^*_{K,n}(1/|x|)} \le |f(x)| \le \eta^*_{K,n}(|x|)
  \end{equation}
  for all $x \in \overline{\R}^n \setminus \{ 0 \}$. Therefore by (\ref{roughupperbound})
  \[\begin{array}{ll}  
    \displaystyle \frac{1}{c_2}|x|^\beta \le |f(x)| \le c_1 |x|^\alpha, &\textnormal{if }\,\, 0<|x|<1,\\
    \displaystyle \frac{1}{\eta^*_{K,n}(1)} \le |f(x)| \le \eta^*_{K,n}(1), & \textnormal{if }\,\, |x| = 1,\\
    \displaystyle \frac{1}{c_1}|x|^\alpha \le |f(x)| \le c_2 |x|^\beta, & \textnormal{if } |x|>1,\\
  \end{array}\]
  for $c_1 = \eta^*_{K,n}(1)\lambda_n^{1-\alpha}$ and $c_2= \eta^*_{K,n}(1)\lambda_n^{\beta-1}$. We can estimate $\max \{ c_1,c_2 \} \le c_3 = \exp (60 \sqrt{K-1})$ for $K \in (1,2]$.
\end{proof}

\begin{remark} In the case $n=2$ the result \cite[Lemma 1]{gl} gives bounds
\[
  \begin{array}{ll}
      \displaystyle \frac{1}{c_4}|x|^\beta \le |f(x)|, & \textnormal{if } 0<|x|\le1,\\
      \displaystyle |f(x)| \le c_4 |x|^\beta, & \textnormal{if } |x|>1,
    \end{array}
\]
where $f$ is as in Proposition \ref{qcestimate} and $c_4 = \exp(7K)$. Proposition \ref{qcestimate} improves this result for small $K$. Since $7K \ge 60 \sqrt{K-1}$ is equivalent to
\[
  K \le \frac{60}{49} \left( 30-\sqrt{851} \right) \approx 1.0139947,
\]
Proposition \ref{qcestimate} gives better result for $K \in (1,1.0399]$.

\end{remark}

\begin{lemma} \label{vesnalemma} If $0<a<1<b$, $q=\sqrt{\frac{b-1}{1-a}}$ and
$m\geq\max\{q,q^{-1}\}$, then
\begin{equation} \label{v1}
 mt^a-t\geq t-\frac{t^b}m
 \end{equation}
holds for $0<t\leq 1$ and
\begin{equation} \label{v2}
 mt^b-t\geq t-\frac{t^a}m
 \end{equation}
holds for $t\geq 1$.
\end{lemma}

\begin{proof}  With $f(t)=mt^{a-1}+\frac{t^{b-1}}{m}$ we see that
the first inequality is equivalent to $f(t)\geq 2$ for $0<t\le 1$.
By differentiating we conclude that the function $f$ is decreasing
for
$$0<t\leq\left(\frac{m^2(1-a)}{b-1}\right)^{1/(b-a)}\equiv t_0.$$
Therefore the function $f$ has a minimum at $t_0.$ By the choice of
$m$ we see that $t_0\ge 1$ and hence for all $t\in (0,1]$
$$f(t)\ge f(1)=m+m^{-1}\ge 2.$$ The first inequality (\ref{v1}) is proved and the proof
for (\ref{v2}) is similar. \end{proof} 

\begin{lemma}\label{c3estimate}
  Let $n \ge2$, $K>1$, $\alpha = K^{1/(1-n)}$, $\beta=1/\alpha$ and $c_3 = \sqrt{\beta}$. For $t \in (0,1)$
  \begin{equation}\label{c3estimate1}
    c_3 t^\alpha -t \ge t-\frac{t^\beta}{c_3}
  \end{equation}
  and for $t > 1$
  \begin{equation}\label{c3estimate2}
    c_3 t^\beta -t \ge t-\frac{t^\alpha}{c_3}.
  \end{equation}
Moreover, both (\ref{c3estimate1}) and (\ref{c3estimate2}) hold for
all constants $c_4 \ge c_3,$ e.g. for $c_4=e^{60\sqrt{K-1}}.$
\end{lemma}
\begin{proof}
For the application of Lemma \ref{vesnalemma} we observe that
$$ \frac{\beta-1}{1-\alpha}= \beta >1.$$
 Now the proof follows from Lemma \ref{vesnalemma} and the
inequalities
$$ e^{60\sqrt{K-1}} \ge e^{\sqrt{\beta-1}}\ge \sqrt{\beta}\,. $$
Note that that the function $h(t)= m \max \{t^{\alpha}, t^{\beta} \}
+ \min \{t^{\alpha}, t^{\beta}\}/m$ is increasing in $m$ when $t$ is
fixed.
\end{proof}

Next we introduce some Bernoulli type inequalities. To prove our main result we need the following lemma, which is new to our knowledge. To some extent it is similar to the Bernoulli type inequalities in \cite[1.58(30)]{avv}. For $a =1= b$ part (4) of the next lemma coincides with the usual Bernoulli inequality
\cite[p. 34(4)]{m}. An appendix at the end on this paper gives some additional Bernoulli type inequalities, not needed for this paper, but which have further been studied in \cite{kmsv}.

\begin{lemma}\label{genbernoulli}
  Let $0 < a \le 1 \le b$ and $\varphi(t) = \max \{ t^a,t^b \}$. Then, with $u=\log^{1-a} 2$,
  \begin{itemize}
    \item[(1)] the function
    \[
      f_1(t) = \frac{\log (1+t)}{\log (1+t^a)}
    \]
    is increasing on $(0,\infty)$ with range $(0,1/a)$.
    \item[(2)] For $t \in (0,1)$
    \[
      u \le f_2(t) < 1,
    \]
    where
    \[
      f_2(t) = \frac{\log (1+t^a)}{\log^a (1+t)}.
    \]
    The function $f_2(t)$ is decreasing on $(0,1)$ and increasing on $(1,\infty)$ with $f_2(1) = u$.
    \item[(3)] The function
    \[
      f_4(t) = \frac{\log (1+t^b)}{\log (1+t)}
    \]
    is increasing on $(0,\infty)$ with range $(0,b)$.
    \item[(4)] For $c > 1$
    \[
      \log(1+c \varphi(t)) \le
      \left\{ \begin{array}{ll} c \log^a (1+t), & 0<t<1,\\ {c}{b} \log (1+t), & t \ge 1. \end{array} \right.
    \]
  \end{itemize}
\end{lemma}
\begin{proof}
  \textit{(1)} We will first show that $f_1(t)$ is an increasing function. By a straightforward computation
  \[
    f_1'(t) = \frac{1}{\log^2(1+t^a)} \left( \frac{\log (1+t^a)}{1+t}-\frac{a \log (1+t)}{t+t^{1-a}} \right)
  \]
  and $f_1'(t) \ge 0$ is equivalent to
  \begin{equation}\label{derivativeoff1}
    a g(t) \le g(t^a)
  \end{equation}
  for $g(t) = (1+1/t)\log(1+t)$. By differentiation we see that $g(t)$ is increasing and therefore (\ref{derivativeoff1}) holds for $t \in (0,1]$.

  Let us assume $t > 1$. We will show that the function $h(t) = g(t^a)-a g(t)$ is decreasing. We obtain
  \[
    h'(t) = \frac{a}{t^{2+a}} \left( t^a \log(1+t)-t \log(1+t^a) \right)
  \]
  and $h'(t) \le 0$ is equivalent to $s(t) \le s(t^a)$ for $s(t) = \frac{\log(1+t)}{t}$. By \cite[p. 273, 3.6.18]{m} $t/(1+t) \le \log (1+t)$ and therefore
  \[
    s'(t) = \frac{\frac{t}{1+t}-\log(1+t)}{t^2} \le 0
  \]
  and implying $s(t) \le s(t^a)$. We conclude that $f_1(t)$ is an increasing function on $(0,\infty)$.

  By the l'Hospital Rule
  \[
    \lim_{t\to 0} f_1(t) = \lim_{t\to 0} \frac{t^{1-a}(1+t^a)}{a(1+t)} = 0
  \]
  and
  \[
    \lim_{t\to \infty} f_1(t) = \lim_{t\to \infty} \frac{t^{1-a}(1+t^a)}{a(1+t)} = \frac{1}{a}
  \]
  and the assertion follows.

  \textit{(2)} We will show that $f_2(t)$ is decreasing on $(0,1)$ and increasing on $(1,\infty)$. We have
  \[
    f_2'(t) = \frac{a}{\log^a (1+t)} \left( \frac{1}{t+t^{1-a}}-\frac{\log(1+t^a)}{(1+t)\log (1+t)} \right)
  \]
  and $f_2'(t) \le 0$ is equivalent to
  \begin{equation}\label{derivativeoff2}
    g(t) \le g(t^a)
  \end{equation}
  for $g(t) = (1+1/t)\log(1+t)$. The function $g(t)$ is increasing on $(0,\infty)$ because $g'(t) = (t-\log(1+t))/t^2 > 0$. Therefore inequality (\ref{derivativeoff2}) is true and $f_2'(t) < 0$ for $t \in (0,1]$. Similarly, $f_2'(t) \ge 0$ for $[1,\infty)$. Now $f_2(t)$ is decreasing on $(0,1)$ and increasing on $(1,\infty)$. Therefore
  \[
    u = f_2(1)  \le f_2(t) < \lim_{t \to 0} f_2(t) = 1.
  \]

  \textit{(3)} We show that $f_4(t)$ is an increasing function on $(0,\infty)$. Since
  \[
    f_4'(t) = \frac{1}{\log^2(1+t)} \left( \frac{b \log(1+t)}{t+t^{1-b}}-\frac{\log(1+t^b)}{1+t} \right)
  \]
  the inequality $f_4'(t) \ge 0$ is equivalent to $b g(t) \ge g(t^b)$ for $g(t) = (1+1/t)\log(1+t)$. It holds true by proof of \textit{(1)}.

  By the l'Hospital Rule
  \[
    \lim_{t\to 0} f_4(t) = \lim_{t\to 0} \frac{b(1+t)t^{b-1}}{(1+t^b)} = 0
  \]
  and
  \[
    \lim_{t\to \infty} f_4(t) = \lim_{t\to \infty} \frac{b(1+t)t^{b-1}}{(1+t^b)} = b
  \]
  and the assertion follows.

  \textit{(4)} The assertion follows since by the Bernoulli inequality \cite[p. 34(4)]{m} and \textit{(2)}
  \[
    \log(1+c t^a) \le c \log(1+t^a) \le c \log^a (1+t)
  \]
  for $t \in (0,1)$, and by the Bernoulli inequality \cite[p. 34(4)]{m} and \textit{(3)}
  \[
    \log(1+c t^b) \le c \log (1+t^b) \le c b \log(1+t)
  \]
  for $t \ge 1.$
\end{proof}

\begin{proposition}
	\label{prop215}
		\textit{(1)}
			Let $u,v\in(0,\infty)$ and $u+v>1>|u-v|$. Then the points of intersection of the circles
			$$
			\left\{
			\begin{array}{l}
			x^2+y^2=v^2,\vspace{0.5em}\\
			(x-e_1)^2+y^2=u^2
			\end{array}
			\right.
			$$
			are $(\frac{1+v^2-u^2}2,\pm w)$, $w=\sqrt{v^2-\left(\frac{1+v^2-u^2}2\right)^2}$.

		\textit{(2)}
			Let $0<s<v$, $0<\varepsilon<v-s$. Then
			$$
			h(s)\equiv \sqrt{v^2-s^2}-\sqrt{v^2-(s+\varepsilon)^2}
			$$
			is increasing on $(0,v-\varepsilon)$ and hence
			$$
			h(s)\leqslant h(v-\varepsilon)=\sqrt{\varepsilon( 2  v- \varepsilon)}
			$$
			for all $s\in(0,v-\varepsilon)$. Moreover,
			$$
			B\equiv|(s,\sqrt{v^2-s^2})-(s+\varepsilon,\sqrt{v^2-(s+\varepsilon)^2})|\leqslant\sqrt{2 \, \varepsilon v}.
			$$
\end{proposition}
\begin{proof}
	\textit{(1)} Obvious.
	
	\textit{(2)} The fact that $h$ is increasing is clear.
	Moreover, by increasing property of $h$
	$$
	 B^2=\varepsilon^2+(\sqrt{v^2-u^2}-\sqrt{v^2-(u+\varepsilon)^2})^2\leqslant\varepsilon^2+\varepsilon(2v-\varepsilon)=2\varepsilon v. \qedhere
	$$
	\end{proof}

\begin{figure}[!ht]
  \begin{center}
    \includegraphics[width=10cm]{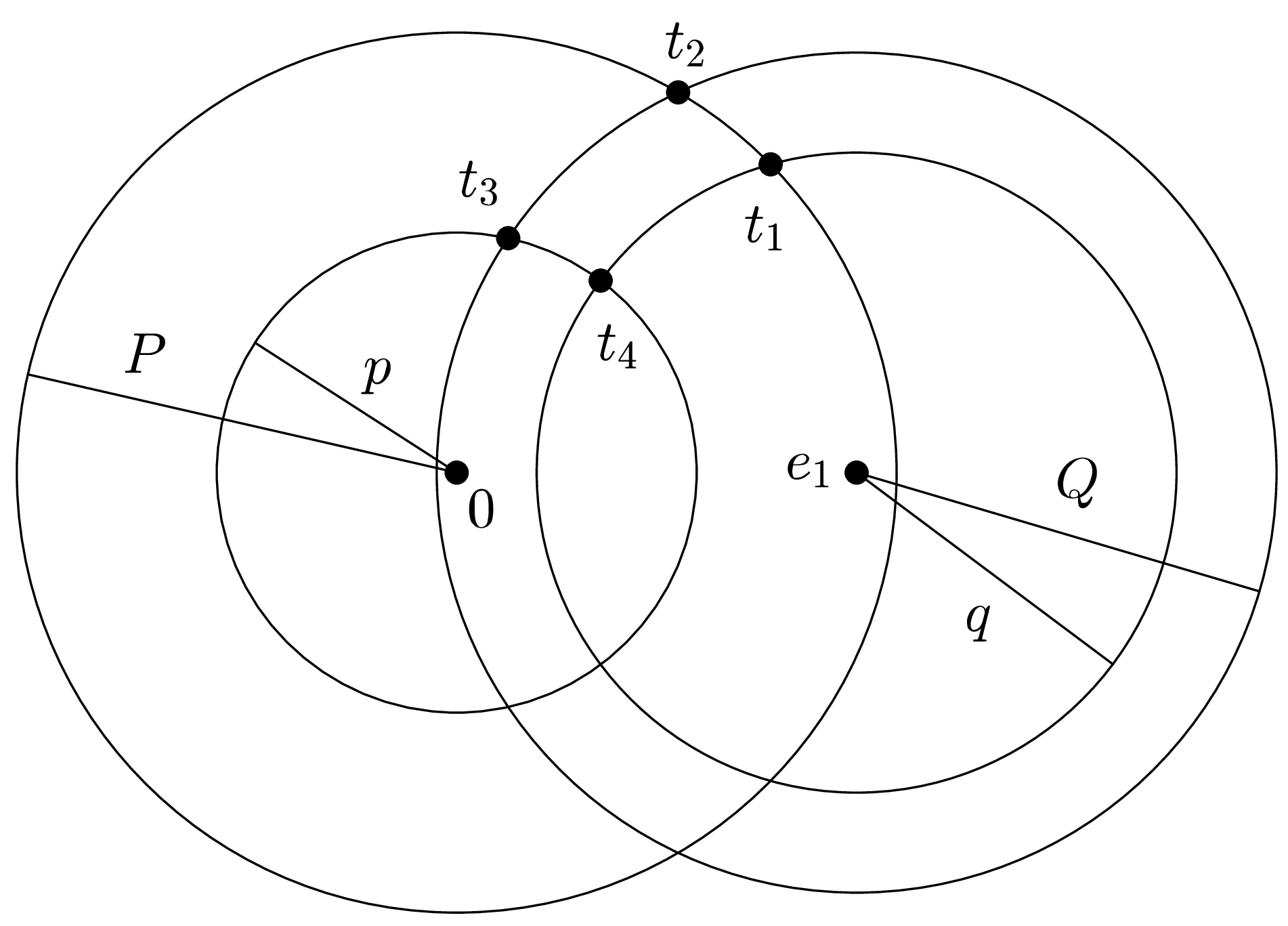}
    \caption{The four circles in Lemma 2.18.\label{fig:lemma218}}
  \end{center}
\end{figure}

\begin{lemma}
	\label{lemma216}
	Let $0<p<P$ and $0<q<Q$ be numbers such that $S(0,p)$ and $S(0,P)$ both intersect $S(e_1,Q)$
	at two points $t_1,t_2$ and also $S(e_1,q)$ at two points $t_3,t_4$. Then the set
	$$
	A\equiv(B^2(0,P)\setminus\bar B^2(0,p))\cap(B^2(e_1,Q)\setminus\bar B^2(e_1,q))\cap\mathbb H^2
	$$
	does not meet the $e_1$-axis. Suppose that the points
	$t_1,t_2,t_3,t_4$ occur in the positive order when we traverse $\partial A\,.$ Then
	\begin{equation}
	\label{bound}
	2\,\diam(A)\le{\frac\pi 2}(|t_1-t_2|+|t_2-t_3|+|t_3-t_4|+|t_4-t_1|)
	\end{equation}
	and, moreover,
	$$
	\begin{array}{l}
	|t_1-t_2|^2\leqslant P(Q^2-q^2),
	\vspace{0.5em}\\
	|t_2-t_3|^2\leqslant Q(P^2-p^2),
	\vspace{0.5em}\\
	|t_3-t_4|^2\leqslant p(Q^2-q^2),
	\vspace{0.5em}\\
	|t_4-t_1|^2\leqslant q(P^2-p^2),
\vspace{0.5em}\\
\diam(A) \le 2 \pi \max\{ P,Q  \} \sqrt{\varepsilon} \,.
	\end{array}
	$$
\end{lemma}
\begin{proof}
	It is clear that $2\,\diam(A)\leqslant L$, where $L$ is the sum of the lengths of the four
	circular arcs with endpoints $t_i,t_j$. These four arcs form the boundary $\partial A$.
	For instance, the length of the first arc is less than $\frac\pi 2|t_1-t_2|$ and similarly
	for other arcs and hence the desired bound (\ref{bound}) follows.
	
	By Proposition \ref{prop215} (1)
	$$
	\begin{array}{ll}
	t_1=(a,\sqrt{P^2-a^2}), & \displaystyle a=\frac{1+P^2-q^2}2,
	\vspace{0.5em}\\
	t_2=(b,\sqrt{P^2-b^2}), & \displaystyle b=\frac{1+P^2-Q^2}2,
	\vspace{0.5em}\\
	t_3=(c,\sqrt{p^2-c^2}), & \displaystyle c=\frac{1+p^2-Q^2}2,
	\vspace{0.5em}\\
	t_4=(d,\sqrt{p^2-d^2}), & \displaystyle d=\frac{1+p^2-q^2}2.
	\end{array}
	$$
	
	Next, we obtain by Proposition \ref{prop215} (2)
	$$
	\begin{array}{ll}
	|t_1-t_2|\leqslant\sqrt{|a-b|P}=\sqrt{P(Q^2-q^2)},
	\vspace{0.5em}\\
	|t_2-t_3|\leqslant\sqrt{|b-c|Q}=\sqrt{Q(P^2-p^2)},
	\vspace{0.5em}\\
	|t_3-t_4|\leqslant\sqrt{|c-d|p}=\sqrt{p(Q^2-q^2)},
	\vspace{0.5em}\\
	|t_4-t_1|\leqslant\sqrt{|a-d|q}=\sqrt{q(P^2-p^2)}.
	\end{array}
	$$
The final inequality follows, because 
$$ |t_1-t_2| \le \sqrt{   P(Q+q) 2 \varepsilon} \le 2 m \sqrt{    \varepsilon}\,, \quad m=  \max\{P,Q\},$$
and similar upper bounds hold for each of the terms on the right hand side of \eqref{bound}.
\end{proof}
	
	For $0<s<r$, $x\in\R^n\,,$ we denote $R(x,r,s)=B^n(x,r)\setminus\bar B^n(x,s)$.

\begin{lemma} \label{finalbd}
Let $x\in\mathbb H^2$ with $|x|+|x-e_1|>1$,
$$\varepsilon\in(0,\min\{ |x|+|x-e_1|-1,  1-||x|-|x-e_1||\}/3)$$
and let
$$
A(\varepsilon)=R(0,|x|+\varepsilon,|x|-\varepsilon)\cap R(e_1,|x-e_1| +\varepsilon,|x-e_1|-\varepsilon)\cap{\mathbb H}^2\,.
$$
Then
$$\diam(A(\varepsilon))\leqslant 2\pi \max\{|x|,|x-e_1|\}\sqrt{\varepsilon}\,.$$
\end{lemma}
\begin{proof}
The bound for $\varepsilon$ implies that the condition $1>|u-v|$ in Lemma \ref{lemma216} is satisfied for
$u=|x| \pm \varepsilon $ and $v=|x-e_1| \pm \varepsilon \,.$
	The proof follows from Lemma \ref{lemma216}.
\end{proof}
\section{Diameter estimate}

In this section we will consider the intersection of the two annuli
 \begin{equation} \label{my31} 
 (B^n(0,b) \setminus  \overline{B}^n(0,a) ) \cap ( B^n(e_1,d) \setminus  \overline{B}^n(e_1,c))\,,
 \end{equation}
when $b>a>0\,,  d>c> 0\,.$ In order to simplify the situation,
we require that this intersection does not contain points of the $e_1$-axis. For this purpose
we introduce a sufficient condition which we call the annuli intersection criterion.
We apply this criterion to study
 a  $K$-quasiconformal mapping $f \colon \overline{\R}^n \to \overline{\R}^n$ with $f(y) = y$ for $y \in \{ 0,e_1,\infty \}$ and our goal is to find an upper bound for quantities such as $|f(x)-x|\,,$ 
  when  $K > 1$ is small enough, the bound for $K$ depending on $n,$ $|x|,|x-e_1|$ in an explicit way via the annuli intersection criterion.

\begin{nonsec} \label{myaic} The annuli intersection criterion \end{nonsec}
We consider the region \eqref{my31} and will give a criterion
for the radii which ensures that this region \eqref{my31} does not contain points of the
$e_1$-axis. Without loss of generality we may assume here that $n=2\,.$ The crucial requirement
is that given a pair of two boundary circles of the four boundary circles of the annuli in
\eqref{my31}, these two circles have a point of intersection. Altogether there are  four points
of intersection, denoted $t_1,t_2,t_3, t_4\,,$ which form the corner points of the region \eqref{my31}\,:
\[
(1) \quad
\{t_1\} = S(0,b) \cap S(e_1, c) \Rightarrow 1-c< b < 1+c
\]
\[
(2) \quad
\{t_2\} = S(0,b) \cap S(e_1, d) \Rightarrow 1-d< b < 1+d
\]
\[
(3) \quad
\{t_3\} = S(0,a) \cap S(e_1, d) \Rightarrow 1-d< a < 1+d
\]
\[
(4) \quad
\{t_4\} = S(0,a) \cap S(e_1, c) \Rightarrow 1-c< a < 1+c
\]
We see that $  (1)  \Rightarrow (2) $ and $  (4)  \Rightarrow (3) $ because $0<c<d\,.$

{\it Conclusion:} For the desired intersection property, it is enough to require
\begin{equation} \label{myaicrit}
1-c < a < 1+c \quad \& \quad 1-c < b < 1+c \,.
\end{equation}
These inequalities constitute the annuli intersection criterion.

A geometric consequence of the annuli intersection criterion is that all the triangles $\Delta(0,1,t_j), j=1,2,3,4\,, $
are non-degenerate, i.e. that the strict triangle inequalities
\[
|t_j|+ |t_j-e_1|>1\,, \quad j=1,2,3,4\,,
\]
\[
1> ||t_j|- |t_j-e_1||\,, \quad j=1,2,3,4
\]
hold.

\begin{nonsec} Application of annuli intersection criterion \label{myaicappl} \end{nonsec}
Fix a point $x \in {\mathbb{R}^n} $ with $ |x|+ |x-e_1|>1 \,.$
We will next find $\delta$ such that with
\[  a=|x|-\delta,  \quad b=|x|+\delta,\quad c=|x-e_1|-\delta, \quad d=|x-e_1|+\delta, \]
the annuli intersection criterion \eqref{myaicrit} is valid. Rewriting
 \eqref{myaicrit} with these $a,b,c,d$ shows that it is enough to choose
 \begin{equation} \label{myaicappl2}
 \delta < \min \{ 1-  ||x|-|x-e_1||, |x|+|x-e_1|-1\}/3 \equiv \E(x)\,.
 \end{equation}
Note that $\E(x) \in (0,1/3)\,.$

\begin{figure}[!ht]
  \begin{center}
    \includegraphics[width=8cm]{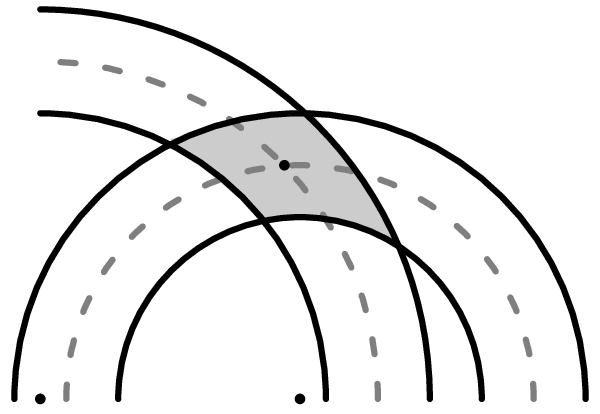}
    \caption{The shaded region is the cross-section of the set $A(\delta)$ in Theorem \ref{epsilonestimate} (1) with a
    half-plane that contains both $0$ and $e_1$.\label{fig:crosssectionA}}
  \end{center}
\end{figure}

\begin{theorem}\label{epsilonestimate}
  Fix $x \in \mathbb{R}^n \setminus \{ 0,e_1 \}$ with $|x|+|x-e_1|>1$ and let $\E(x) $ be as in \eqref{myaicappl2}.

(1) If $\delta \in (0, \E(x))\,,$ then the intersection of annuli
\[
A(\delta) \equiv (B^n(0,|x|+\delta) \setminus  \overline{B}^n(0,|x|-\delta) ) \cap ( B^n(e_1,|x-e_1|+ \delta) \setminus  \overline{B}^n(e_1,|x-e_1|-\delta))\,,
\]
does not contain points of the $e_1$-axis. Moreover, if $\mathbb{H} = \{ x \in \mathbb{R}^n: x_2>0, x_3=\dots = x_n= 0\}\,,$ then
\[
\diam(\mathbb{H} \cap A(\delta)) \le 2 \pi \max \{  |x|,|x-e_1|\} \sqrt{\delta }\,.
\]

(2) If $\delta \in (0, \E(x))\,,$
\[
    1 < K \le  1+ \left( \frac{\log (\e/2+1)}{62} \right)^2 <1 + 10^{-5} \,,
  \]
and  $f \in QC_K(\mathbb{R}^n)\,$ with $f(z)=z$ for $z \in  \{0, e_1\}\,,$  
then  $f(x) \in A(\delta)\,.$ Moreover, there exists a M\"obius transformation, a rotation $h: \mathbb{R}^n \to  \mathbb{R}^n$
around the $e_1$-axis, such that
$$   |h(f(x))-x| \le 2 \pi \max \{  |x|,|x-e_1|\} \sqrt{\delta } \,. $$
\end{theorem}
\begin{proof} (1) Fix $z \in A(\delta)\,.$ The choice of $\Delta(x)$ implies $|z|+|z-e_1|>1\,$ and that the annuli intersection criterion
\eqref{myaicrit} holds for $a=|x| - \varepsilon , b=|x| + \varepsilon, c= |x-e_1| - \varepsilon, d= |x-e_1| + \varepsilon\,.$ 
The upper bound for $\diam(\mathbb{H} \cap A(\delta))$ follows directly from Lemma \ref{finalbd}. 

(2) We need to prove that both
 \[
    |x|-\e \le |f(x)| \le |x|+\e \quad \textrm{and} \quad |x-e_1|-\e \le |f(x)-e_1| \le |x-e_1|+\e
  \]
hold. It is enough to prove the first inequality because the proof of the second one is similar.
Let us denote $l(x) = c_3^{-1} \min \{ |x|^\alpha,|x|^\beta \}$ and $u(x) = c_3 \max \{ |x|^\alpha,|x|^\beta \}$
  where $c_3= e^{60\sqrt{K-1}} $. Then by Proposition \ref{qcestimate}
\[
l(x) \le |f(x)| \le u(x)
\]
for all $x \in {\mathbb{R}^n}\,.$


We will first consider the case $0 < |x| < 1$. By Lemma \ref{c3estimate}
  \begin{eqnarray*}
    \max \{ u(x)-|x|,|x|-l(x) \} & = & \max \left\{ c_3 |x|^\alpha-|x|,|x|-\frac{1}{c_3}|x|^\beta \right\}\\
     & = & c_3 |x|^\alpha-|x|\\
    & \le & \exp(60\sqrt{K-1}) |x|^\alpha - |x|\\
    & \le & \exp(60\sqrt{K-1}) |x|^{1/K} - |x|\\
    & \le & \exp(60\sqrt{K-1})-1.
  \end{eqnarray*}
  Now $\exp(60\sqrt{K-1})-1 \le \e$ is equivalent to
  \begin{equation}\label{Kupperbound1}
    K \le \left( \frac{\log (\e+1)}{60} \right)^2+1.
  \end{equation}

  If $|x| = 1$, then
  \[
    \max \{ u(x)-|x|,|x|-l(x) \} = c_3-1
  \]
  and therefore we want $\exp(60\sqrt{K-1})-1 \le \e$, which is equivalent to
  \begin{equation}\label{Kupperbound2}
    K \le \left( \frac{\log (\e+1)}{60} \right)^2+1.
  \end{equation}

  Let us then consider the case $1 < |x| < 2$. By Lemma \ref{c3estimate}
  \begin{eqnarray*}
    \max \{ u(x)-|x|,|x|-l(x) \} & = & \max \left\{ c_3 |x|^\beta-|x|,|x|-\frac{1}{c_3}|x|^\alpha \right\}\\
    & = & c_3 |x|^\beta-|x|\\
    & \le & \exp(60\sqrt{K-1}) |x|^\beta - |x|\\
    & \le & \exp(60\sqrt{K-1}) |x|^K - |x|\\
    & \le & |x|(\exp(60\sqrt{K-1})|x|^{K-1}-1)\\
    & \le & 2(\exp(60\sqrt{K-1}+(K-1)\log |x|)-1)\\
    & \le & 2(\exp(62 \sqrt{K-1})-1).
  \end{eqnarray*}
  Now $2(\exp(62 \sqrt{K-1})-1) \le \e$ is equivalent to
  \begin{equation}\label{Kupperbound3}
    K \le \left( \frac{\log (\e/2+1)}{62} \right)^2+1.
  \end{equation}
  By combining (\ref{Kupperbound1}), (\ref{Kupperbound2}) and (\ref{Kupperbound3}) we have
  \[
  |x|-\e \le l(x) \le |f(x)|\le u(x) \le |x|+\e
  \]
  for
\begin{eqnarray*}
    K & \le & \min \left\{ \left( \frac{\log (\e+1)}{60} \right)^2+1,2,\left( \frac{\log (\e/2+1)}{62} \right)^2+1 \right\}\\
    & = & \min \left\{ \left( \frac{\log (\e/2+1)}{62} \right)^2+1 ,2 \right\}= 1+ \left( \frac{\log (\e/2+1)}{62} \right)^2 <1 + 10^{-5} \,.
\end{eqnarray*}
A similar argument also implies $|x-e_1|-\e \le |f(x)-e_1| \le |x-e_1|+\e$ and the assertion follows.
Finally, the rotation $h$ is taken to be the rotation around the $e_1$-axis mapping $f(x)$ to the plane
determined by the  $e_1$-axis and the point $x\,.$
\end{proof}





\section{The main results}

In this section we will apply the diameter estimates to obtain results for distortion results in terms of the
$j$-metric and
the quasihyperbolic metric.
For the convenience of the reader we recall some basic properties of this metric.
For a given pair of points $x,y\in G,$ the infimum in \eqref{qhmetric} is always
attained \cite{go}, i.e., there always exists a quasihyperbolic
geodesic $J_G[x,y]$ which minimizes the quasihyperbolic length,
$k_G(x,y)=\ell_k(J_G[x,y])\,$ and furthermore
the distance is additive on the geodesic: $k_G(x,y)=$
$k_G(x,z)+k_G(z,y) $ for all $z\in J_G[x,y]\,.$

\begin{lemma} \label{ksvbd} (\cite[Lemma 3.7 (2)]{vu1}, \cite[Lemma 2.1]{gp})\,. Fix $\lambda \in (0,1) \,.$
For $x,y \in G= {\mathbb R}^n \setminus \{ 0\}$ with
$|x-y|\le \lambda |x|$ we have
\[j_G(x,y) \le k_G(x,y) \le c_1(\lambda) j_G(x,y) \]
with $c_1(\lambda) =1/(1-\lambda).$
\end{lemma}

Only in rare special cases there is a formula for the quasihyperbolic distance
between two points. One such case is when $x,y \in G= {\mathbb R}^n \setminus \{ 0\}\,.$
Martin and Osgood \cite{mo} proved that
\begin{equation}
\label{moformula}
k_G(x,y) = \sqrt{ \log^2 \frac{|x|}{|y|} + \left(2 \, \mathrm{arcsin} \left( \frac{1}{2} \left( \left| \frac{x}{|x|} -\frac{y}{|y|} \right| \right) \right) \right)^2} \,,
\end{equation}
for all $x,y\in G\,.$

\begin{proof}[Proof of Theorem \ref{main1}]
 The proof follows from Theorem \ref{epsilonestimate}.
 \end{proof}

We now prove Theorem \ref{main2}, which can be easily  generalized for arbitrary subdomains of $\mathbb{R}^n \,.$

\begin{proof}[Proof of Theorem \ref{main2}]
  Both sides of the claim are invariant under a homothety mapping $z \mapsto t z, t>0,$
   and so is the normalization. Moreover, composition with a homothety does not change
   $K$-quasiconformality. Therefore
   we may assume that $x = e_1, f(e_1) =e_1\,$ and by symmetry we, furthermore,
   assume that $|y| \ge 1$. Now
  \[
    \frac{|f(y)-f(e_1)|}{|f(e_1)|} = |f(y)-f(e_1)| \le \eta \left( |e_1-y| \right)
  \]
  and
  \[
    \frac{|f(y)-f(e_1)|}{|f(y)|} = \frac{|f(y)-f(e_1)|}{|f(y)-f(0)|} \le \eta \left( \frac{|e_1-y|}{|y-0|} \right) = \eta \left( \frac{|e_1-y|}{|y|} \right).
  \]
  Therefore by Proposition \ref{qcestimate} and Lemma \ref{genbernoulli} (4)
  \begin{eqnarray*}
    j(f(e_1),f(y)) & = & \log \left( 1+\frac{|f(e_1)-f(y)|}{\min \{ |f(e_1)|,|f(y)| \}} \right)\\
    & \le & \log
    \left( 1+\max \left\{ \eta(|y-e_1|),\eta \left( \frac{|x-y|}{|y|} \right) \right\} \right)\\
    & = & \log (1+\eta(|y-e_1|))\\
    & \le & \log (1+c_3 \max \{ |y-e_1|^\alpha,|y-e_1|^{1/\alpha} \})\\
    & \le & \left\{ \begin{array}{ll} c_3 \log^\alpha (1+|y-e_1|), & 0<|y-e_1|<1,\\ \frac{c_3}{\alpha} \log (1+|y-e_1|), & |y-e_1| \ge 1. \end{array} \right.
  \end{eqnarray*}
  By choosing $c(K) = c_3/\alpha$, where $c_3=\exp (60 \sqrt{K-1})$ is as in Lemma \ref{c3estimate}, we have $c(K) \to 1$ as $K \to 1$ and the assertion follows.
\end{proof}

\begin{corollary}
  Let $G \subset \Rn$ be a domain, $f \in QC_K(\mathbb{R}^n)$, $K \in (1,2]$ and $f(z) = z$ for all $z \in \partial G$. There exists $c(K)$ such that for all $x,y \in G$
  \[
    j_G(f(x),f(y)) \le c(K) \max \{ j_G(x,y)^\alpha , j_G(x,y) \},
  \]
  where $\alpha = K^{1/(1-n)}$, and $c(K) \to 1$ as $K \to 1$.
\end{corollary}
\begin{proof}
  We may assume $0,e_1 \in \partial G$ and $d(x) = |x| \le d(y) \le |y|$. Now for $z \in \partial G$
  \[
    \frac{|f(y)-f(x)|}{|z-f(x)|} = \frac{|f(y)-f(x)|}{|f(z)-f(x)|} \le \eta \left( \frac{|x-y|}{|z-x|} \right) \le \eta \left( \frac{|x-y|}{|x|} \right)
  \]
  and
  \[
    \frac{|f(y)-f(x)|}{|z-f(y)|} = \frac{|f(y)-f(x)|}{|f(z)-f(y)|} \le \eta \left( \frac{|x-y|}{|z-y|} \right) \le \eta \left( \frac{|x-y|}{|y|} \right) \le \eta \left( \frac{|x-y|}{|x|} \right).
  \]
  Therefore by Proposition \ref{qcestimate} and Lemma \ref{genbernoulli} (4)
  \begin{eqnarray*}
    j(f(x),f(y)) & = & \log \left( 1+\frac{|f(x)-f(y)|}{\displaystyle \min_{z \in \partial G} \{ |f(x)-z|,|f(y)-z| \}} \right)\\
    & \le & \log \left( 1+\eta \left( \frac{|y-x|}{|x|} \right) \right)\\
    & \le & \log \left( 1+c_3 \max \left\{ \frac{|y-x|^\alpha}{|x|^\alpha},\frac{|y-x|^{1/\alpha}}{|x|^{1/\alpha}} \right\} \right) \\
    & \le & \left\{ \begin{array}{ll} c_3 \log^\alpha \left( 1+\frac{|y-x|}{|x|} \right), & 0<|y-x|<1,\\ \frac{c_3}{\alpha} \log \left( 1+\frac{|y-x|}{|x|} \right), & |y-x| \ge 1. \end{array} \right.
  \end{eqnarray*}
  By choosing $c(K) = c_3/\alpha$, where $c_3=\exp (60 \sqrt{K-1})$ is as in Lemma \ref{c3estimate}, we have $c(K) \to 1$ as $K \to 1$ and the assertion follows.
\end{proof}

\begin{proof}[Proof of Theorem \ref{main3}]
Fix $\lambda \in (0,1/2)\,.$ Let $c\ge1$ be the constant of Theorem
\ref{main2} and write
\[
\mu(\lambda) = \frac{1}{c} \min \{ \log(1/(1-\lambda)), \log^{\beta}(1/(1- \lambda)) \}
= \frac{1}{c}  \log^{\beta}(1/(1-\lambda))\,, \quad \beta =1/ \alpha\,.
\]
Note that $c \to 1$ when $K \to 1\,.$ The rest of the proof is now divided into two cases.

{\bf Case A.} $k_G(x,y)\le \mu(\lambda)\,.$ By Theorem \ref{main2} and Lemma \ref{ksvbd} we have
\begin{eqnarray} \label{mybd1}
\frac{|f(x)-f(y)|}{|f(x)|} & \le & \exp \left( {c} \max \{ j_G(x,y)^{\alpha}, j_G(x,y) \} \right) -1 \nonumber\\
& \le & \exp \left( {c} \max \left\{ \frac{\log^{\beta}(1/(1-\lambda))}{c} , \frac{\log(1/(1-\lambda))}{c}  \right\} \right)-1\nonumber\\ & = & \lambda/(1-\lambda)\,.
\end{eqnarray}
Therefore we can use Lemma \ref{ksvbd} and Theorem \ref{main2} to find an upper bound for $k_G(f(x),f(y))$ in terms of $k_G(x,y)$ and $\lambda$. By Lemma \ref{ksvbd},
(\ref{mybd1}), and Theorem \ref{main2} we have
\begin{eqnarray*}
  k_G(f(x),f(y)) & \le & c_1(\lambda) j_G(f(x),f(y)) \le c c_1(\lambda) \max \{ j_G(x,y)^{\alpha}, j_G(x,y) \} \\
  & \le & c c_1(\lambda)   \max \{ k_G(x,y)^{\alpha}, k_G(x,y) \} \,.
\end{eqnarray*}

{\bf Case B.} $k_G(x,y)\ge \mu(\lambda)\,.$ Choose points $x_j,j=0,...,p+1,$ on the
quasihyperbolic geodesic segment $J_G[x,y]$ joining $x$ with $y$ such that $x_0 =y,
x_{p+1}=x$ and $k_G(x_{j+1},x_j)= \mu(\lambda)$ for $j=0,...,p-1\,,$ $k_G(x_{p+1},x_p) \le
\mu(\lambda)\,.$ Because the quasihyperbolic distance is additive on the geodesic we see
that
\[
k_G(x,y)= \sum_{j=0}^p k_G(x_{j+1},x_j) \ge p \mu(\lambda) \Rightarrow p \le k_G(x,y)/\mu(\lambda)\,.
\]
Next, by the triangle inequality and by Case A
\[
k_G(f(x),f(y))\le \sum_{j=0}^p  k_G(f(x_{j+1}),f(x_j)) \le c c_1(\lambda)
\sum_{j=0}^p  k_G(x_{j+1},x_j)^{\alpha} =V\,.
 \]
Next, using the property of weighted mean values \cite[p. 76, Theorem 1]{m}
\[ \sum_{j=0}^p u_j^{\alpha} \le
(p+1)^{1-\alpha} \left( \sum_{j=0}^p u_j \right) ^{\alpha}\, \textrm{ for\,  all \, positive}\,\, u_j \] we have
\begin{eqnarray*}
  V & \le & c c_1(\lambda)(p+1)^{1- \alpha} \left(\sum_{j=0}^p k_G(x_{j+1},x_j) \right)^{\alpha}\\
  & \le &  c c_1(\lambda) \mu(\lambda)^{\alpha -1 } \left( 1+ \frac{\mu(\lambda)}{k_G(x,y)} \right) ^{1-\alpha} k_G(x,y)\\
  & \le & c c_1(\lambda) \mu(\lambda)^{\alpha -1 } 2^{1-\alpha} k_G(x,y) \,.
\end{eqnarray*}

 In view of the above Cases A and B we see that in both cases we have the claim with the
 constant
 \[
 \omega(K,n) = c c_1(\lambda) \mu(\lambda)^{\alpha -1 } 2^{1-\alpha}   \,.
 \]
 It remains to prove that $ \omega(K,n) \to 1$ when $K \to 1\,.$ It suffices to show that
 we can choose $\lambda$ depending on $K,n$ such that
 \[
\mu(\lambda)^{\alpha -1 } = c^{1- \alpha} \log^{1- \beta}(1/(1-\lambda)) \to 1
 \]
 when $K \to 1\,.$ For instance the  choice $\lambda = \beta-1$ will do.
 \end{proof}

Agard and Gehring have studied the angle distortion under quasiconformal mappings
of the plane \cite{ag}. Motivated by their work we record the following corollary
of Theorem \ref{main3}.

\begin{corollary} Suppose that under the hypotheses of Theorem \ref{main3} $x, y \in
S^{n-1}$ and $f(x), f(y) \in S^{n-1}$ and let $\phi$ and $\psi$ be the angles between
the segments $[0,x], [0,y]$ and $[0,f(x)],[0, f(y)]\,,$ respectively. Then
\[
\psi \le \omega(K,n) \max \{ \phi^{\alpha}, \phi \} \,.
\]
\end{corollary}

\begin{proof}
The proof follows easily from the Martin-Osgood formula (\ref{moformula})\,.
\end{proof}

\begin{remark}
  It is well-known (see \cite{lvv},\cite[Lemma 4.28]{avv2}) that for a given $K \ge 1$ there exists a $K$-quasiconformal mapping $f \colon \Rn \to \Rn$ with $f(0)=0$ such that $f(-e_1) = -e_1$ and $f(e_1) = \lambda(K^{1/(n-1)})e_1$, where $\lambda(K) = \varphi_{K}(1/\sqrt{2})^2/(1-\varphi_{K}(1/\sqrt{2})^2)$ (see \cite[p. 203, 204]{avv}). Choosing $G = \R^n \setminus \{ 0 \}$ and $x = 1 =-y$ in Theorem \ref{main2} we see that
\[
  \log (1+L) \le c(K)\log 3,
\]
where $L=(1+\lambda(K^{1/(n-1)}))$, and hence the constant $c(K)$ has to satisfy
\[
  c(K) \ge \frac{\log(2+\lambda(K^{1/(n-1)}))}{\log 3}.
\]
In order to compare this estimate to the upper bound in Theorem \ref{main2} well-known estimates for $\lambda(K^{1/(n-1)})$ may be used. For instance we know that $\lambda(K) \ge \exp (\pi(K-1))$ by \cite[Corollary 10.33]{avv}.

This same idea can be applied to produce a lower bound for the constant $\omega(K,n)$
of Theorem \ref{main3} as well.
\end{remark}

\textbf{Acknowledgement.} The second author wishes to acknowledge the suggestions offered by Vladimir Bozin.

\small

\normalsize

\appendix

\section{Appendix: Bernoulli type inequalities}

In Lemma \ref{genbernoulli} we introduced Bernoulli type inequalities. In this appendix we introduce some additional Bernoulli type inequalities.

\begin{lemma}
  Let $0 < a \le 1 \le b$ and $\varphi(t) = \max \{ t^a,t^b \}$. Then, with $u=\log^{1-a} 2$, $v=\log^{1-b} 2$,
  \begin{itemize}
    \item[(1)] For $t \in (0,\infty)$
    \[
      f_3(t) = \frac{\log (1+t^b)}{\log^b (1+t)} \le v.
    \]
    The function $f_3(t)$ is increasing on $(0,1)$ and decreasing on $(1,\infty)$ with $f_3(1) = v$.
    \item[(2)] For $t \in (0,\infty)$
    \[
      v^{-1} \log(1+\varphi(t)) \le \varphi(\log (1+t)).
    \]
    \item[(3)] For $t \in (0,\infty)$
    \[
      \varphi(\log(1+t)) \le \left\{ \begin{array}{ll} u^{-1} \log(1+\varphi(t)), & t \in (0,e-1], \\ \log^b (1+\varphi(t)), & t>e-1. \end{array} \right.
    \]
    \item[(4)] For $t \in (0,t_0)$, $t_0 > 0$ and
    \[
      f_8(t) = \frac{\log (1+\varphi(t))}{\varphi(\log(1+t))}
    \]
    we have $\min \{ u, f(t_0) \} \le f(t) \le \log(1+(e-1)^b)$.
    \item[(5)] For $s,t > 0$
    \[
      2^{1-b} \le \frac{\varphi(s)+\varphi(t)}{\varphi(s+t)} \le 2^{1-a}.
    \]
  \end{itemize}
\end{lemma}
\begin{proof}
  \textit{(1)} Follows from Lemma \ref{genbernoulli} \textit{(2)} by choosing $b = 1/a$.

  \textit{(2)} If $t \in (0,1]$ then by Lemma \ref{genbernoulli} \textit{(2)}
  \[
    \log(1+\varphi(t)) = \log(1+t^a) \le \log^a(1+t) = \varphi(\log(1+t))
  \]
  and if $t \in [1,e-1]$ then by \textit{(1)}
  \[
    \log(1+\varphi(t)) = \log(1+t^b) \le v \log^b(1+t) \le v \log^a(1+t) =  v \varphi(\log(1+t)).
  \]
  If $t \ge e-1$ then by \textit{(1)}
  \[
    \log(1+\varphi(t)) = \log(1+t^b) \le v \log^b (1+t) = v \varphi(\log(1+t))
  \]
  and the assertion follows.

\textit{(3)} If $t \in (0,1]$ then by Lemma \ref{genbernoulli} \textit{(2)}
  \[
    \varphi(\log(1+t)) = \log^a(1+t) \le u^{-1} \log (1+t^a) = u^{-1} \log(1+\varphi(t)).
  \]
  and if $t \in [1,e-1]$ then by Lemma \ref{genbernoulli} \textit{(2)}
  \[
    \varphi(\log(1+t)) = \log^a(1+t) \le u^{-1} \log (1+t^a) \le u^{-1} \log(1+t^b) = u^{-1} \log(1+\varphi(t)).
  \]
  If $t \ge e-1$ then by \textit{(1)}
  \[
    \varphi(\log(1+t)) = \log^b (1+t) \le \log^b (1+t^b) = \log^b (1+\varphi(t)).
  \]
  and the assertion follows.

  \textit{(4)} For $t \in (0,1]$, Lemma \ref{genbernoulli} \textit{(2)} implies that $u \le f(t) \le 1$ and for $t \ge e-1$, \textit{(1)} implies that $f(t_0) \le f(t) \le f(e-1) = \log (1+(e-1)^b)$.

  Let us consider the case $t \in (1,e-1)$. Now $f'(t) \ge 0$ is equivalent to
  \begin{equation}\label{f8estimate}
    \frac{\log(1+t^b)}{\log(1+t)} \le \frac{b(1+t)}{a t (1+t^{-b})}.
  \end{equation}
  We have
  \[
    \frac{\log(1+t^b)}{\log(1+t)} \le b \le \frac{b(1+t)}{t (1+t^b)} \le \frac{b(1+t)}{a t (1+t^{-b})},
  \]
  where the first inequality follows from Lemma \ref{genbernoulli} \textit{(1)} and the second inequality holds, because $t(1+t^{-b}) \le 1+t$ is equivalent to $t^{1-b} \le 1$ which holds true by the selection of $b$ and $t$. Now \eqref{f8estimate} holds and $f(t)$ is increasing on $(1,e-1)$. Thus $u = f(1) \le f(t) \le f(e-1)$. The assertion follows by combining the three cases.

  \textit{(5)} Follows from \cite[1.58 (27) p. 340]{avv}.
\end{proof}

{\rm\small

\medskip

({\it Riku Kl\'en, Matti Vuorinen}) {\sc Department of Mathematics and Statistics, University of Turku, 20014 Turku, Finland}

{\it E-mail address}: \verb"riku.klen@utu.fi, vuorinen@utu.fi"

\medskip

({\it Vesna Todor\v{c}evi\'c}) {\sc Mathematical Institute of the Serbian Academy of Sciences and Arts,
Faculty of Organizational Sciences, University of Belgrade, Serbia}

{\it E-mail address}: \verb"vesnat@fon.bg.ac.rs"

}

\end{document}